\documentclass{amsart}

\usepackage{amssymb,amsmath,amsthm,latexsym,amscd,amsfonts,mathrsfs,mathtools}

\usepackage{hyperref}
\usepackage{cleveref}
\usepackage{epsfig}
\usepackage{psfrag}
\usepackage{graphicx}
\usepackage{color}
\usepackage[margin=3.55cm]{geometry}

\numberwithin{equation}{section}
\newtheorem{theorem}{Theorem}[section]

\newtheorem{proposition}[theorem]{Proposition}
\newtheorem{corollary}[theorem]{Corollary}
\newtheorem{lemma}[theorem]{Lemma}
\newtheorem{fact}[theorem]{Fact}
\theoremstyle{definition}
\newtheorem{definition}[theorem]{Definition}

\newtheorem{remark}[theorem]{Remark}
\newtheorem{notation}[theorem]{Notation}
\newtheorem{example}[theorem]{Example}

\newcommand{\tr}{\mathrm{tr}}

\newcommand{\C}{\mathbb{C}}

\newcommand{\rar}{\rightarrow}

\begin{document}
\title[A short note on relative entropy]{A short note on relative entropy for a pair of intermediate subfactors}
\subjclass[2010]{46L37}

\author[K C Bakshi]{Keshab Chandra Bakshi} \address{Chennai
  Mathematical Institute, Chennai, INDIA} \email{bakshi209@gmail.com, kcbakshi@cmi.ac.in}

\thanks{The first named author was supported through a DST
  INSPIRE Faculty grant (reference  no. DST/INSPIRE/04/2019/002754).}
\date{9 May, 2021}
\maketitle
\begin{abstract}
 Given a quadruple of finite index subfactors we explicitly compute the Pimsner-Popa probabilistic constant for the pair of intermediate subfactors and relate
 it with the corresponding Connes-St\o rmer relative entropy between them. This generalizes an old result of Pimsner and Popa.
\end{abstract}

\section{Introduction}
Generalizing the classical notion of (conditional) entropy from ergodic theory, Connes and St\o rmer in \cite{CS} defined a relative entropy
$H(B_1|B_2)$ between a pair of finite dimensional $C^*$-subalgebras of a finite von Neumann algebra $M$ with a fixed faithful normal trace $\tr$. Using this they obtained an appropriate  definition of the entropy of an automorphism in the non-commutative framework of operator algebras. As an application, they proved that for $n\neq m$ the $n$-shift of the hyperfinite $II_1$ factor is not conjugate to the $m$-shift. Furthermore, they also proved a Kolmogorov-Sinai type theorem using the relative entropy as the main technical tool. However, in an impactful paper \cite{PP}, Pimsner and Popa
had observed that the definition of the relative entropy does not depend on $B_1,B_2$ being finite dimensional, so that one may also consider the relative entropy
$H(B_1|B_2)$ for arbitrary von Neumann subalgebras $B_1,B_2\subset M$. In another direction, as a generalization of the index of a subgroup in a group, Jones in a groundbreaking paper \cite{Jo} introduced a notion of index $[M:N]$ as the Murray-von Neumann’s coupling constant $\text{dim}_{N}(L^2(M))$  for any subfactor $N\subset M$ of type $II_1$. Quite surprisingly, Pimsner and Popa had discovered that if $N$ and $M$ are type $II_1$ factors and $N\subset M$ then $H(M|N)$ depends on both the Jones index  and the relative commutant. More precisely, 
they proved (among other things) that if the relative commutant is trivial, that is $N^{\prime}\cap M=\C$, then
\begin{equation}
 \label{pimsnerpopa1}
 H(M|N)=\log [M:N].
\end{equation}
In fact, Pimsner and Popa had shown that the equality holds in \Cref{pimsnerpopa1} if and only if the subfactor $N\subset M$ is extremal.

Given the fact that subfactor theory deals with the relative position of a subfactor inside an ambient factor, it is a very natural and fundamental question to consider relative positions of multiple subfactors. For an irreducible subfactor with finite Jones index the set of its intermediate subfactors forms a finite lattice under two natural operations. By a Galois correspondence, every interval sublattice of finite groups can be realized as intermediate subfactor lattice of finite index and thereby the study of intermediate subfactor lattice of finite index may be thought of as a natural generalization of the study of interval sublattice of finite groups. Realization of even some simple finite lattices as intermediate (irreducible) subfactor lattices remains entirely open. We want to understand this lattice better. To this end, recently in \cite{BDLR},  we discovered a new notion of angle between two intermediate subfactors $P$ and $Q$ of the finite index subfactor $N\subset M$ of type $II_1$ factors. In this paper we have taken yet another approach by investigating the Connes-Stormer relative entropy and it looks like more calculable in the non-irreducible setting. We consider a pair of intermediate subfactors
$N\subset P,Q\subset M$ of a finite index subfactor $N\subset M$ of type $II_1$ factors with $N^{\prime}\cap M=\C$ (and thus form a quadruple which we denote by $(N,P,Q,M)$), and obtain a formula for $H(P|Q)$  in terms of a probabilistic number $\lambda(P,Q)$ due to Pimsner and Popa.

For von Neumann subalgebras $B_2\subset B_1\subset M$ of a finite von Neumann algebra $M$, Pimsner and Popa, in \cite{PP}, defined  the probabilistic constant 
\begin{equation}\label{probabilistic index}\lambda(B_1,B_2)=\text{max}\{\lambda>0|E_{B_2}(x)\geq \lambda x, x\in B_{1_{+}}\}.
\end{equation}This serves as a replacement of Jones index when $B_1$ and $B_2$ are not necessarily factors. Quite remarkably, they also prove that if
$M$ is a type $II_1$ factor and $N\subset M$ is a subfactor then
\begin{equation}
 \label{pimsnerpopa2}
 {\lambda(M,N)}={[M:N]}^{-1}.
\end{equation}

Interestingly, the definition of $\lambda(B_1,B_2)$, as in  \Cref{probabilistic index}, works for general subalgebras $B_1$ and $B_2$ (not necessarily, $B_2\subset B_1)$ 
of $M$ as well; and furthermore, one may also consider the number $\lambda(B_2,B_1).$ 

As a generalization of \Cref{pimsnerpopa2}, we prove the following result:
\smallskip

\noindent{\bf \Cref{commuting}.} {Let $(N,P,Q,M)$ be a quadruple of type $II_1$ factors with $[M:N]<\infty$. 
 \begin{enumerate}
  \item If $(N,P,Q,M)$ is a commuting square then $\lambda(P,Q)={[P:N]}^{-1}$ and $\lambda(Q,P)={[Q:N]}^{-1}.$
  \item If $(N,P,Q,M)$ is a co-commuting square then $\lambda(P,Q)={[M:Q]}^{-1}$ and $\lambda(Q,P)={[M:P]}^{-1}.$
 \end{enumerate}}
\smallskip
 Moreover, in \Cref{whenequalityholds} and \Cref{whenequalityholds2} we prove that the converse of \Cref{commuting} also holds true for an irreducible quadruple (i.e., $N^{\prime}\cap M=\C$). 

In this article, we also provide some calculable formulae for the probabilistic  numbers in the case of an irreducible quadruple.\smallskip

\noindent{\bf \Cref{imp1}.}{
 Suppose $(N,P,Q,M)$ is a quadruple of $II_1$ factors with $N^{\prime}\cap M=\C$ and $[M:N]<\infty$ and $e_P$ (resp. $e_Q$) is the biprojection corresponding to the intermediate subfactor $P$ (resp. $Q$), then $\lambda(Q,P)= \frac{\text{tr}(e_Pe_Q)}{\text{tr}(e_Q)}$ and
 $\lambda(P,Q)= \frac{\text{tr}(e_Pe_Q)}{\text{tr}(e_P)}.$}
 \smallskip

Finally, we prove the main result of this article by relating Connes-St\o rmer relative entropy $H(P|Q)$ with $\lambda(P,Q)$. 

\smallskip

\noindent{\bf \Cref{main}.}{Let $(N,P,Q,M)$ be an irreducible quadruple such that $[M:N]<\infty$. Then, 
 $H(Q|P)=-\log\big(\lambda(Q,P)\big).$}
 
The above formula generalizes \Cref{pimsnerpopa1}.\smallskip

Combining \Cref{main} with \Cref{imp1} , we deduce the following formula for the relative entropy (in \Cref{imp2}):
 \begin{equation*}
  H(Q|P)= \log\big(\tr(e_Q)\big)-\log\big(\tr(e_Pe_Q)\big).\end{equation*}

We conclude the paper with a cute application of \Cref{main}, where we deduce (in \Cref{prop2}) that if $(N,P,Q,M)$ is an irreducible quadruple with $[M:N]<\infty$ and $[P:N]=2$ then $H(P,Q)$ is either $0$ or $\log 2.$

The paper is organized as follows. After some Preliminaries in Section 2, we discuss the probabilistic constants and present some useful formulae for them in Section 3 and finally, 
in Section 4, we prove our main result involving the Connes-St\o rmer relative entropy.
\section{Preliminaries}
In this section, we fix the notations and recall some results from \cite{BDLR}, which we will use frequently in the sequel.
\begin{notation}\label{n}
In this paper we only deal with separable type $II_1$ factors. Consider a subfactor $N\subset M$ of a type $II_1$ factor $M$ with $[M:N]<\infty$. Throughout we will be dealing only with subfactors of type
$II_1$ with finite Jones index. Thus, $M$ has the unique faithful, normal  tracial state $\tr$.
\begin{enumerate}
\item 
A quadruple $$\begin{matrix}
  Q &\subset & M \cr \cup &\ &\cup\cr N &\subset & P,
\end{matrix}$$ denoted by $(N,P,Q,M)$, is called irreducible if $N^{\prime}\cap M=\C$. Consider the basic
 constructions $N\subset M \subset M_1$, $P \subset M \subset P_1$ and
 $Q \subset M \subset Q_1$. As is standard, we denote by $e_1$ the
 Jones projection $e^M_N$. It is easily seen that, as $II_1$-factors
 acting on $L^2(M)$, both $ P_1$ and $Q_1$ are contained in $ M_1$. In
 particular, if $e_P: L^2(M) \rar L^2(P)$ denotes the orthogonal
 projection, then $e_P \in M_1$. Likewise, $e_Q \in M_1$. Note that, $\tr(e_P)={[M:P]}^{-1}.$ Thus, we
 naturally obtain a dual quadruple $$\begin{matrix} P_1 &\subset & M_1
   \cr \cup &\ &\cup\cr M &\subset & Q_1. 
\end{matrix}$$ We call $(M,Q_1,P_1,M_1)$ the basic construction of $(N,P,Q,M)$. 

\item A quadruple $(N,P,Q,M)$ is called a commuting square if
$E^M_P E^M_Q= E^M_Q E^M_P = E^M_N$.
 A quadruple $(N,P,Q,M)$ is called a co-commuting square if the quadruple $(M,Q_1,P_1,M_1)$ is a commuting square.

\item Suppose $N_{-1}\subset N\subset M$ is a downward basic construction. Also, denote by $P_{-1}$ (resp. $Q_{-1}$) a downward basic construction of $N\subset P$ (resp. $N\subset Q$) with
the corresponding Jones projection $e^N_{P_{-1}}$ (resp. $e^N_{Q_{-1}}$).
We obtain a new quadruple
$$\begin{matrix} P_{-1} &\subset & N
   \cr \cup &\ &\cup\cr N_{-1} &\subset & Q_{-1}. 
\end{matrix}$$
We call this new quadruple $(N_{-1},Q_{-1},P_{-1},N)$ as a downward basic construction of the quadruple $(N,P,Q,M)$.
\end{enumerate}

\end{notation}

 We recall without proofs the following elementary facts.
\begin{fact}\label{fact}
Consider a quadruple of type $II_1$ factors $(N,P,Q,M)$ with $[M:N]<\infty.$
\begin{enumerate}
 \item  Suppose $(M,Q_1,P_1,M_1)$ is the basic construction of the quadruple and let $e_{P_1} (\text{resp.}~e_{Q_1})$
 be the Jones projection for the inclusion $P_1\subset M_1 (\text{resp.}~ Q_1\subset M_1).$ Then,
 
$$\tr(e_{P_1}e_{Q_1})=\frac{[M:P]}{[Q:N]}\tr(e_Pe_Q)=\frac{[M:Q]}{[P:N]} \tr(e_Pe_Q).$$
\item Suppose $(N_{-1},Q_{-1},P_{-1},N)$ is a downward basic construction of $(N,P,Q,M)$. Then,
$$\tr(e^N_{P_{-1}}e^N_{Q_{-1}})=\frac{[M:Q]}{[P:N]}\tr(e_Pe_Q)=\frac{[M:P]}{[Q:N]} \tr(e_Pe_Q).$$

\item $[P_{-1}:N_{-1}]=[M:P]$ and $[Q_{-1}:N_{-1}]=[M:Q].$
\end{enumerate}

\end{fact}

\begin{definition}\cite[Definition 2.17]{BDLR}\label{auxiliary}
Consider a quadruple $(N,P,Q,M)$ as in \Cref{fact}. Let $\{\lambda_i:i\in I\}$ and
$\{\mu_j:j\in J\}$ be (right) Pimsner-Popa bases for $P/N$ and $Q/N$,
respectively. Define two auxiliary opertors $p(P,Q)$ and $p(Q,P)$
 as follows:
\[
p(P,Q)= \sum_{i,j}{\lambda_i}\mu_j e_1 {\mu}^*_j{\lambda}^*_i\quad \text{and}\quad
p(Q,P)= \sum_{i,j}\mu_j \lambda_i e_1 {\lambda}^*_i {\mu}^*_j.
\]\end{definition}
Below we recall from \cite{BDLR} a few important facts about the auxiliary operators $p(P,Q)$ and $p(Q,P)$.
\begin{fact} \label{fact2} Consider an irreducible  quadruple of type $II_1$ factors $(N,P,Q,M)$ with $[M:N]<\infty$. \begin{enumerate}
 \item  $p(P,Q)$ and $p(Q,P)$ are both
independent of choice of bases. See \cite[Lemma 2.18]{BDLR} for details.\smallskip

\item $p(P,Q)= [P:N] E^{N^{\prime}}_{P^{\prime}}(e_Q)=[Q:N]E^{M_1}_{Q_1}(e_P).$ Similarly,
$p(Q,P)= [Q:N] E^{N^{\prime}}_{Q^{\prime}}(e_P)=[P:N]E^{M_1}_{P_1}(e_Q).$ This follows from \cite[Proposition 2.25]{BDLR}. \smallskip

\item By \cite[Proposition
  2.22]{BDLR}, $Jp(P,Q)J = p(Q,P)$, where $J$ is the usual modular
conjugation operator on $L^2(M)$; so that, $\|p(P,Q)\| =
\|p(Q,P)\|$.  \smallskip

\item $\lambda:= \lVert p(P,Q)\rVert=\lVert p(Q,P)\rVert.$ We note that $\lambda=[M:N] \tr(e_Pe_Q)$. 

Indeed, by 
\cite[Remark 3.3]{BDLR}, first observe that $\lambda=[M:Q] {\tr}\big(p(P,Q)e_Q\big)$ and then, by the proof of \cite[Proposition 3.5]{BDLR}, we note that ${\tr}\big(p(P,Q)e_Q\big)=[Q:N] {\tr}(e_Pe_Q).$ 
\end{enumerate}
\end{fact}
\color{black}
We also recall the following result which will be used heavily in this note.

\begin{lemma}\cite{BDLR}\label{bdlr}
 \label{crucial}
 If $N^{\prime}\cap M=\C$ then $\frac{1}{\lambda} p(P,Q)$ is a projection and $\frac{1}{\lambda}p(P,Q)\geq e_P\vee e_Q.$ Similar statement holds if we interchange $P$ and $Q$.
\end{lemma}
\begin{remark}\label{downward}
 The auxiliary operators corresponding to a downward basic construction (see \Cref{n}) $(N_{-1},Q_{-1},P_{-1},N)$ of the quadruple $(N,P,Q,M)$ satisfy the following formulae:

$$p(Q_{-1},P_{-1})=[Q_{-1}:N_{-1}]E^{N^{\prime}_{-1}}_{Q^{\prime}_{-1}}(e^N_{P_{-1}})=[P_{-1}:N_{-1}]E^M_P(e^N_{Q_{-1}})$$ and 
 $$p(P_{-1},Q_{-1})=[P_{-1}:N_{-1}]E^{N^{\prime}_{-1}}_{P^{\prime}_{-1}}(e^N_{Q_{-1}})=[Q_{-1}:N_{-1}]E^M_Q(e^N_{P_{-1}}).$$\end{remark}

\section{Pimsner-Popa probabilistic constant}
Generalizing the Jones index, Pimsner and Popa in \cite[Notation 2.5]{PP} had introduced the probabilistic constant $\lambda(B_1,B_2)$ for von Neumann subalgebras $B_2\subset B_1\subset M$ of a finite von Neumann algebra $M$
and this proved to be a powerful analytical tool in subfactor theory. However, as observed in {\cite[Definition 4.1]{O}}, this definition works for general subalgebras $B_1$ and $B_2$ (not necessarily, $B_2\subset B_1)$ as well.
In this section we shall obtain a formula for $\lambda(P,Q)$, for  a pair of intermediate subfactors $P$ and $Q$ of a subfactor $N\subset M$.
\begin{definition}(Pimsner-Popa)\label{pimsner-popa}
 Consider a pair of von Neumann subalgebras $P$ and $Q$ of a finite von Neumann algebra $M$. The Pimsner-Popa probabilistic constant of the ordered pair $(P,Q)$  is defined as follows:
 $$\lambda(P,Q)=\text{max}\{t>0| E_Q(x)\geq t x ~~\forall x\in P_{+}\}.$$
\end{definition}
\begin{remark}\label{remarkone}
 If $N\subset M$ is a  subfactor of type $II_1$ factors then $\lambda(N,M)=1.$ This follows easily from the above definition.\end{remark}
In \cite[Theorem 2.2]{PP}, Pimsner and Popa have given a very useful characterization of Jones index. More precisely, they obtained the following relationship between $[M:N]$ and $\lambda(M,N)$ for any subfactor $N\subset M$ of type $II_1$ factors:
\begin{equation}\label{pimsner}
{[M:N]}^{-1}=\lambda(M,N).\end{equation}
 Below we consider a quadruple $(N,P,Q,M)$ and obtain a calculable formula for $\lambda(P,Q)$ (and also of $\lambda(Q,P)$) in the case $N\subset M$ is irreducible. If $Q=M$ and $P=N$, we recover \Cref{pimsner}.
We hope to obtain formulae for $\lambda(P,Q)$ and $\lambda(Q,P)$ in the general case as well in a future publication.
\begin{theorem}\label{imp1}
 Suppose $(N,P,Q,M)$ is a quadruple of $II_1$ factors with $N^{\prime}\cap M=\C$, then $\lambda(Q,P)= \frac{{\tr}(e_Pe_Q)}{{\tr}(e_Q)}$ and
 $\lambda(P,Q)= \frac{{\tr}(e_Pe_Q)}{{\tr}(e_P)}.$
 \end{theorem}
\begin{proof}
 Let $x\in P_{+}$.  Since $x$ is positive, we must have
 $$x=\big(\sum_j a_je^N_{P_{-1}}b_j\big)^*\big(\sum_i a_ie^N_{P_{-1}}b_i\big)=\sum_{i,j}b^*_jE^N_{P_{-1}}(a^*_ja_i)e^N_{P_{-1}}b_i,$$ where $a_i,b_j\in N$.
 By \Cref{crucial} and \Cref{fact2} [item (4)], we have
 $$\frac{1}{[N:N_{-1}]\tr(e^N_{P_{-1}}e^N_{Q_{-1}})}p(P_{-1},Q_{-1})\geq e^N_{P_{-1}}$$ and thanks to \Cref{fact} we have $$\frac{1}{[M:P][M:Q]\tr(e_Pe_Q)}[Q_{-1}:N_{-1}]E^{M}_{Q}(e^N_{P_{-1}})\geq e^N_{P_{-1}}.$$
 Again by \Cref{fact}, $[Q_{-1}:N_{-1}]=[M:Q]$  so that
 \begin{equation}\label{inequality}
\frac{\tr(e_P)}{\tr(e_Pe_Q)}E^M_Q(e^N_{P_{-1}})\geq e^N_{P_{-1}}.
\end{equation}
Thus, we deduce that : \begin{align*}
      x  = & \sum_{i,j} b^*_jE^N_{P_{-1}}(a^*_ja_i)e^N_{P_{-1}}b_i\\
      \leq & \frac{\tr(e_P)}{\tr(e_Pe_Q)}\sum_{i,j} b^*_jE^N_{P_{-1}}(a^*_ja_i)E^M_Q(e^N_{P_{-1}})b_i   ~~~~~~\text{[By \Cref{inequality}]}\\
      \leq & \frac{\tr(e_P)}{\tr(e_Pe_Q)}E^M_Q\bigg(\sum_{i,j} b^*_jE^N_{P_{-1}}(a^*_ja_i)e^N_{P_{-1}}b_i\bigg)\\
      \leq & \frac{\tr(e_P)}{\tr(e_Pe_Q)}E^M_Q(x).
     \end{align*}
In other words, for any $x\in P_{+}$, we have
\begin{equation}\label{inequality2} E_Q(x)\geq \frac{\tr(e_Pe_Q)}{\tr(e_P)}x.
\end{equation}

Now, suppose $E^M_Q(x)\geq sx$ for some $s>0$. Then, $E^M_Q(e^N_{P_{-1}})\geq s e^N_{P_{-1}}$. Taking norm on both sides we get, $\lVert E^M_Q(e^N_{P_{-1}})\rVert\geq s$.
    By \Cref{fact2} and \Cref{downward} we see that $$\lVert E^M_Q(e^N_{P_{-1}})\rVert =\frac{[M:N]\tr(e^N_{P_{-1}}e^N_{Q_{-1}})}{[M:Q]}.$$ Hence, by \Cref{fact},
\begin{equation}\label{inequality3}
\frac{\tr(e_Pe_Q)}{\tr(e_P)}\geq s.\end{equation} Therefore, combining \Cref{inequality2} and \Cref{inequality3} we get $ \lambda(P,Q)=\frac{\tr(e_Pe_Q)}{\tr(e_P)},$ as desired.

The other formula follows by interchanging $P$ and $Q$.

 This completes the proof. \end{proof}
\begin{remark}
 Suppose $(N,P,Q,M)$ be as in \Cref{imp1}. In general, $\lambda(P,Q) \neq \lambda(Q,P)$ as can be easily seen from \Cref{remarkone}. Using \Cref{imp1}, it follows that $\lambda(P,Q)=\lambda(Q,P)$ if and only if $[P:N]=[Q:N].$
\end{remark}

\begin{example}\label{ex-1}
Let $G$ be a finite group acting outerly on a $II_1$-factor $S$. Let
$H, K $ and $L$ be subgroups of $G$ such that $L \subseteq H \cap K$
and $H$ and $K$ are non-trivial.   Consider the quadruple $(N = S
\rtimes L, P = S \rtimes H, Q = S \rtimes K, M= S \rtimes G )$. Then, a simple calculation shows that (see \cite[Section 2.2]{BG}, for instance) \begin{equation*}
  \tr(e_P e_Q)  = \frac{ |H \cap K|}{|G|}, \tr(e_P)=\frac{|H|}{|G|} ~~~\text{and}~~~\tr(e_P)=\frac{|K|}{|G|}.
  \end{equation*}
  Therefore, $$\lambda(P,Q)=\frac{|H\cap K|}{|H|}~~~\text{and}~~~\lambda(Q,P)=\frac{|H\cap K|}{|K|}.$$
\end{example}
\begin{example}
 Let $H, K, G$ and $S$ be as in \Cref{ex-1}.
Consider the quadruple $(N= S^G, P = S^H, Q = S^K, M = S)$. Then,
$$\lambda(P,Q)=\frac{|H\cap K|}{|K|}~~~~\text{and}~~~\lambda(Q,P)=\frac{|H\cap K|}{|H|}.$$ The proof is simple and omitted.
\end{example}

\begin{corollary}\label{dual}
 Let $(N,P,Q,M)$ be a quadruple of $II_1$ factors with $N^{\prime}\cap M=\C$ and $[M:N]<\infty$. Then, $\lambda(P_1,Q_1)=\lambda(Q,P)$ and $\lambda(Q_1,P_1)=\lambda(P,Q).$
\end{corollary}
\begin{proof}
 It is easy to check that $\tr(e_{Q_1})=\frac{1}{[Q:N]}.$ Thus, using \Cref{fact} and \Cref{imp1}, we see that
 $$\lambda(Q_1,P_1)=\frac{\tr(e_{P_1}e_{Q_1})}{\tr(e_{Q_1})}=[M:P] \tr(e_Pe_Q)=\lambda(P,Q).$$
 Similarly, $\lambda(P_1,Q_1)=\lambda(Q,P).$ This completes the proof.
\end{proof}
\begin{corollary}\label{downwarddual}
Let $(N,P,Q,M)$ be a quadruple of $II_1$ factors with $N^{\prime}\cap M=\C$ and $[M:N]<\infty$. Then, $\lambda(P_{-1},Q_{-1})=\lambda(Q,P)$ and $\lambda(Q_{-1},P_{-1})=\lambda(P,Q).$
 \end{corollary}
 \begin{proof}
  First, it is easy to check that $\tr(e^N_{P_{-1}})={[P:N]}^{-1}.$ Thus by \Cref{imp1} and \Cref{fact}, 
  \begin{equation*}\lambda(P_{-1},Q_{-1})=\frac{\tr(e^N_{P_{-1}}e^N_{Q_{-1}})}{\tr(e^N_{P_{-1}})}=[M:Q]\tr(e_Pe_Q)=\lambda(Q,P).
  \end{equation*}
  Similarly, $\lambda(Q_{-1},P_{-1})=\lambda(P,Q).$ This completes the proof.
 \end{proof}

\begin{corollary}\label{whenequalityholds}
 Let $(N,P,Q,M)$ be a quadruple of type $II_1$ factors with $N^{\prime}\cap M=\C$ and $[M:N]<\infty$. Then, $1\geq \lambda(P,Q)\geq {[P:N]}^{-1}.$ Furthermore, 
 $\lambda(P,Q)= {[P:N]}^{-1}$ if and only if 
 $(N,P,Q,M)$ is a commuting square. Also, $\lambda(P,Q)=1$ if and only if $P\subset Q$. A similar statement holds if we interchange $P$ and $Q$.
 \end{corollary}
\begin{proof}  From \Cref{pimsner-popa} it follows immediately that $\lambda(P,Q)\leq \lambda(N,Q)$ and also that $\lambda(N,Q)=1$. Thus, $1\geq \lambda(P,Q).$ To show the inequality  $\lambda(P,Q)\geq {[P:N]}^{-1}$, first observe that $\tr(e_Pe_Q)\geq \frac{1}{[M:N]}$. Indeed, $e_Q\geq e_N$ and hence $e_Pe_Qe_P\geq e_N$ and now taking trace on both sides of the last inequality we get $\tr(e_Pe_Q)\geq \tr(e_N)=\frac{1}{[M:N]}$, as desired. Hence, by \Cref{imp1} , we obtain 
$$\lambda(P,Q)=\frac{\tr(e_Pe_Q)}{\tr (e_P)}\geq \frac{1}{[M:N]} [M:P]=\frac{1}{[P:N]}.$$ We conclude that if $(N,P,Q,M)$ is a quadruple of type $II_1$ factors, then $1\geq \lambda(P,Q)\geq {[P:N]}^{-1}.$
 
 Consider the auxiliary operator $p(P,Q)$ as in \Cref{auxiliary} and following \Cref{fact2} , put $\lVert p(P,Q)\rVert=\lambda.$ 
Since, by \Cref{fact2} [item (4)], $\lambda=[M:N] \tr(e_Pe_Q))$, by \Cref{imp1}, we obtain
\begin{equation}\label{lambda}
\lambda(P,Q)=\frac{\lambda}{[P:N]}.
\end{equation}

Recall, given a quadruple of $II_1$ factors $(N,P,Q,M)$, a notion of (interior) angle between $P$ and $Q$, denoted by $\alpha^N_M(P,Q)$, was introduced in \cite{BDLR}. By \cite[Proposition 2.4]{BDLR}, $(N,P,Q,M)$ is a commuting square if and only if ${\alpha}^N_M(P,Q)=\pi/2$. Furthermore, by \cite[Proposition 4.5]{BG}, ${\alpha}^N_M(P,Q)=\pi/2$ if and only if $\lambda=1$. In other words, $(N,P,Q,M)$ form a commuting square  if and only if $\lambda=1$ if and only if $\lambda(P,Q)={[P:N]}^{-1}~~~\text{(by \Cref{lambda})}.$
 
 If $P\subset Q$ then $e_P\leq e_Q$ and hence using \Cref{imp1} we see that $\lambda(P,Q)=1.$ Conversely, if $\lambda(P,Q)=1$ then we get $e_P=e_Pe_Q=e_Qe_P$ and hence $P=P\cap Q.$ Therefore,
 we get $P\subset Q$.
 
 This completes the proof. 
\end{proof}
\begin{corollary}
 \label{whenequalityholds2}
 Let $(N,P,Q,M)$ be a quadruple of type $II_1$ factors with $N^{\prime}\cap M=\C$ and $[M:N]<\infty$. Then, $1\geq \lambda(P,Q)\geq {[M:Q]}^{-1}.$ Furthermore, $(N,P,Q,M)$ is a co-commuting square if and only if $\lambda(P,Q)={[M:Q]}^{-1}.$
 A similar statement holds if we interchange $P$ and $Q$.
\end{corollary}
\begin{proof}
 That $\lambda(P,Q)\geq {[M:Q]}^{-1}$ follows from \Cref{pimsner-popa} . Indeed, as 
 $$\{t >0: E_Q(x)\geq tx ~~\forall~~ x\in M_{+}\}\subset \{t >0: E_Q(x)\geq tx ~~\forall~~ x\in P_{+}\},$$ it follows from \Cref{pimsner-popa} that $\lambda(M,Q)\leq \lambda(P,Q)$ and by \Cref{pimsner} we know that $\lambda(M,Q)={[M:Q]}^{-1}.$

  Recall, $(N,P,Q,M)$ is a co-commuting square if and only if $(M,Q_1,P_1,M_1)$ is a commuting square. Therefore, by \Cref{whenequalityholds} , $(N,P,Q,M)$ is a co-commuting square if and only if
 $\lambda(Q_1,P_1)={[Q_1:M]}^{-1}={[M:Q]}^{-1}$ if and only if $\lambda(P,Q)={[M:Q]}^{-1}$ (thanks to \Cref{dual}). This proves the assertion.
\end{proof}
\begin{proposition}\label{prop}
 Let $(N,P,Q,M)$ be a quadruple of type $II_1$ factors with $N^{\prime}\cap M=\C$ and $[M:N]<\infty$. If $[P:N]=2$, then $\lambda(P,Q)$ is either $1$ or $1/2.$
\end{proposition}
\begin{proof}
 Let $\{1,u\}$ be a unitary orthonormal Pimsner-Popa basis for $P/N$ (which exists by \cite[Corollary 3.4.3]{Jo}). Thus, $p(P,Q)=e_Q+ue_Qu^*$. By \Cref{bdlr}, $p(P,Q)e_Q=\lambda e_Q$ and hence $\big(1+ uE_Q(u^*)\big)e_Q=\lambda e_Q.$
 Thus, by \cite[Lemma 1.2]{PP}, we obtain \begin{equation}\label{final}\lambda-1=  uE_Q(u^*).\end{equation} Now two cases arise. 
 If $\lambda=1$ then $(N,P,Q,M)$ is a commuting square (as already observed in the proof of \Cref{whenequalityholds})and hence, by \Cref{whenequalityholds}, $\lambda(P,Q)=1/2.$ When $\lambda\neq 1$, by \Cref{final}, $u\in Q$ and therefore $P\subset Q$. 
 So, by \Cref{whenequalityholds} again, we conclude that $\lambda(P,Q)=1.$ 
\end{proof}
By \Cref{whenequalityholds} we know that if $N\subset M$ is irreducible then $(N,P,Q,M)$ is a commuting square if and only if $\lambda(P,Q)={[P:N]}^{-1}$. Similarly, \Cref{whenequalityholds2} says that if $N\subset M$ is irreducible then $(N,P,Q,M)$ is a co-commuting square if and only if $\lambda(P,Q)={[M:Q]}^{-1}$. Unfortunately, we are not able to prove above formulae for the general case (i.e., without assuming irreducibility). However, as a partial progress we have the following result.
\begin{theorem}\label{commuting}
 Let $(N,P,Q,M)$ be a quadruple of type $II_1$ factors with $[M:N]<\infty$. 
 \begin{enumerate}
  \item If $(N,P,Q,M)$ is a commuting square then $\lambda(P,Q)={[P:N]}^{-1}$ and $\lambda(Q,P)={[Q:N]}^{-1}.$
  \item If $(N,P,Q,M)$ is a co-commuting square then $\lambda(P,Q)={[M:Q]}^{-1}$ and $\lambda(Q,P)={[M:P]}^{-1}.$
 \end{enumerate}

\end{theorem}
\begin{proof}
 If $x\in P_{+}$, we have, by commuting square condition, $E^M_Q(x)=E^M_Q\circ E^M_P(x)=E^P_N(x)$. Thanks to \cite[Proposition 2.1]{PP}, we know that $E^P_N(x) \geq {[P:N]}^{-1} x$ and so $E^M_Q(x)\geq  {[P:N]}^{-1} x$.
 Now, if $E^M_Q(x)\geq tx$ for all $x\in P_{+}$ and for some scalar $t$, we see that $E^M_Q(e^N_{P_{-1}})= E^P_N(e^N_{P_{-1}})={[P:N]}^{-1}$ and so ${[P:N]}^{-1}\geq te^N_{P_{-1}}$.
 Taking norm to the both sides of the last equation we  get ${[P:N]}^{-1}\geq t.$ Therefore, $\lambda(P,Q)= {[P:N]}^{-1}.$ 
 Similarly, $\lambda(Q,P)={[Q:N]}^{-1}.$ This completes the proof of item (1).

We now prove item(2). Since $(N,P,Q,M)$ is a co-commuting square it is easy to see that $(N_{-1},Q_{-1},P_{-1},N)$ is a commuting square. According to \cite[Proposition 2.20]{BDLR},
$p(P_{-1},Q_{-1})$ is a projection such that $p(P_{-1},Q_{-1})\geq e^N_{P_{-1}}$. By \Cref{downward} and \Cref{fact} we see that 
\begin{equation}\label{commsquare}
[M:Q]E^M_Q(e^N_{P_{-1}})\geq e^N_{P_{-1}}.
\end{equation}\color{black}
For any $x\in P_{+}$, as in the proof of \Cref{imp1}, we see that $x= \sum_{i,j} b^*_jE^N_{P_{-1}}(a^*_ja_i)e^N_{P_{-1}}b_i$ for some $a_i,b_j\in N$. Therefore,
\begin{align*}
 x  \leq  & [M:Q]\sum_{i,j} b^*_jE^N_{P_{-1}}(a^*_ja_i)E^M_Q(e^N_{P_{-1}})b_i~~~~~~\text{[Using \Cref{commsquare}]}\\
 =& [M:Q] E^M_Q\big(\sum_{i,j}b^*_jE^N_{P_{-1}}(a^*_ja_i)e^N_{P_{-1}}b_i\big)\\
 =& [M:Q]E^M_Q(x).
\end{align*}
Thus, we have proved that for any $x\in P_{+}$ we have $E^M_Q(x)\geq {[M:Q]}^{-1} x$. Also, if $E^M_Q(x)\geq sx$ for $x\in P_{+}$ and for some scalar $s$ then,
in particular, $E^M_Q(e^N_{P_{-1}})\geq s e^N_{P_{-1}}$ and hence $p(P_{-1},Q_{-1})\geq [M:Q] s e^N_{P_{-1}}$. Thus, $\lVert p(P_{-1},Q_{-1})\rVert =1\geq s[M:Q].$
So, ${[M:Q]}^{-1}\geq s.$ We conclude that $\lambda(P,Q)={[M:Q]}^{-1}$ and similarly $\lambda(Q,P)={[M:P]}^{-1}.$ This completes the proof.
\end{proof}
\begin{remark}
 Since the quadruple $(N,M,N,M)$ is  both a commuting and a co-commuting square, in view of the fact that $\lambda(M,N)={[M:N]}^{-1}$ (see \cite{PP}[Theorem 2.2]), \Cref{commuting} can be thought of as a natural generalization of \cite{PP}[Theorem 2.2].
\end{remark}
\color{black}
\begin{example}
 Let $K\subset L$ be a subfactor of finite index and $G$ be a finite group acting outerly on $L$ so that we obtain a quadruple $(N=K, P=L,Q=K\rtimes G,M=L\rtimes G).$ Then 
 the quadruple is a commuting square. Thus, $\lambda(L,K\rtimes G)=[L:K]$ and $\lambda(K\rtimes G,L)=|G|.$ 
\end{example}

 \begin{remark}
 We feel that $\lambda(P,Q)$ (and $\lambda(Q,P)$) is a powerful invariant in determining the relative position between the intermediate subfactors $N\subset P,Q\subset M$, and demands a deeper investigation.
\end{remark}

\section{Relative entropy and intermediate subfactors}
Connes and St\o rmer in \cite{CS} introduced a notion of entropy of a finite dimensional subalgebra and more generally, the relative
entropy between two finite dimensional subalgebras of a finite von Neumann algebra as an extension of the notion of entropy from classical ergodic theory. 
Pimsner and Popa in \cite{PP} had observed that this notion does not depend on the fact that the subalgebras are finite-dimensional and they further extended it to a notion of relative entropy between intermediate von Neumann subalgebras of a finite von Neumann algebra.
We briefly recall the definition below.
\begin{definition}(Connes-St\o rmer)\label{cs}
  Let $\eta:[0,\infty)\rightarrow (-\infty,\infty)$ be defined by $\eta(t)=-t\log(t)$ for $t>0$ and $\eta(0)=0.$  Let $M$ be a finite von Neumann algebra with a fixed normalized trace $\tr$ and $P$ and $Q$ be von Neumann subalgebras of $M$.
The entropy of $P$ relative to $Q$ with respect to $\tr$ is 
$$H_{\tr}(P|Q)=\text{sup}\sum_i\bigg(\tr\big(\eta(E_Q(x_i))\big)-\tr\big(\eta(E_P(x_i))\big)\bigg),$$ where the supremum is taken over all finite partitions of unity $1=\sum_ix_i$ in $M$, and $E_P$ and $E_Q$ are the $\tr$-preserving conditional expectations on $P$ and $Q$, respectively.
If $M$ is type $II_1$ factor we often suppress $\tr$ in the notation for the relative entropy as the trace is uniquely determined in this case.
  \end{definition}
 Below, we see that the Pimsner-Popa probabilistic constant is closely related with the relative entropy. Recall, in \cite[Corollary 4.6]{PP}, Pimsner and Popa had proved that if $N\subset M$ is irreducible then 
 \begin{equation}\label{con}
 H(M|N)=-\log \lambda(M,N).
 \end{equation}
More generally, we consider a quadruple of type $II_1$ factors $(N,P,Q,M)$ with $[M:N]< \infty$ and $N^{\prime}\cap M=\mathbb{C}$ and obtain a formula of $H(P|Q)$ (and also of $H(Q|P)$) which generalizes \Cref{con}. This is the main result of this article.
 First we need a lemma.
 \begin{lemma}
  \label{liu}
  Let $(N,P,Q,M)$ be an irreducible quadruple such that $[M:N]<\infty$. Then,
 $$\tr\bigg(\eta\big(E^M_P(e^N_{Q_{-1}})\big)\bigg)=-\frac{1}{[Q:N]}\log\big(\lambda(Q,P)\big).$$
 \end{lemma}
\begin{proof}Following \Cref{n} , suppose $(N_{-1},Q_{-1},P_{-1},N)$ is a downward basic construction of $(N,P,Q,M)$ and $p(Q_{-1},P_{-1})$ is the corresponding auxiliary operator.
 By \Cref{downward}, 
 \begin{equation}\label{e1}
 \tr\bigg(\eta\big(E^M_P(e^N_{Q_{-1}})\big)\bigg)=\tr\bigg(\eta\big(\frac{1}{[P_{-1}:N_{-1}]}p(Q_{-1},P_{-1}))\big)\bigg).
 \end{equation}
 Now, using \Cref{fact2} (item (4)) and \Cref{bdlr} we see that $\frac{1}{[M:N]\tr(e^N_{P_{-1}}e^N_{Q_{-1}})}p(Q_{-1},P_{-1})$ is a projection. Denote this projection by $f$. By \Cref{fact},
 $$f=\frac{1}{[M:P][M:Q]\tr(e_Pe_Q)}p(Q_{-1},P_{-1}).$$
 Since $p(Q_{-1},P_{-1})=[Q_{-1}:N_{-1}] E^{N^{\prime}_{-1}}_{Q^{\prime}_{-1}}(e^N_{P_{-1}})$, it follows that $$\tr(p(Q_{-1},P_{-1})=\frac{[Q_{-1}:N_{-1}]}{[N:P_{-1}]}=\frac{[M:Q]}{[P:N]},$$ and so
 \begin{equation} \label{tracef} \tr(f)= \frac{1}{[M:N]tr(e_Pe_Q)}.\end{equation}
 Since $[P_{-1}:N_{-1}]=[M:P]$, \Cref{e1} implies that 
 $$\tr\bigg(\eta\big(E^M_P(e^N_{Q_{-1}})\big)\bigg)= \tr\bigg(\eta\big(\frac{\tr(e_Pe_Q)}{\tr(e_Q)}f\big)\bigg).$$
 It follows easily (see \cite{JLW}) that if $\alpha$ is a scalar then $\eta(\alpha f)=\eta(\alpha) f.$ Therefore, by \Cref{tracef} we obtain
 $$\tr\bigg(\eta\big(E^M_P(e^N_{Q_{-1}})\big)\bigg)=\tr\bigg(\eta\big(\frac{\tr(e_Pe_Q)}{\tr(e_Q)}\big)f\bigg)=\frac{1}{[M:N]\tr(e_Pe_Q)}\eta\big(\frac{\tr(e_Pe_Q)}{\tr(e_Q)}\big).$$
In other words, $$\tr\bigg(\eta\big(E^M_P(e^N_{Q_{-1}})\big)\bigg)=-\frac{1}{[Q:N]}\log\big(\frac{\tr(e_Pe_Q)}{\tr(e_Q)}\big).$$ The proof is now complete once we apply \Cref{imp1}.
 \end{proof}

\begin{theorem}\label{main}
 Let $(N,P,Q,M)$ be an irreducible quadruple such that $[M:N]<\infty$. Then, 
 $H(Q|P)=-\log\big(\lambda(Q,P)\big).$
\end{theorem}
\begin{proof}
To establish  $H(Q|P)\leq -\log\big(\lambda(Q,P)\big)$ (this inequality generalizes an observation of \cite[Proposition 4.1]{O}) we suitably modify the proof of \cite[Proposition 3.5]{PP}.  We provide sufficient details for the sake of self-containment. First note that for any $x\in M_{+}$ we must have
 $$E_P\big(E_Q(x)\big)\geq \lambda(Q,P) E_Q(x).$$
 Since $\log$ is operator increasing we get
 $${\big(E_Q(x)\big)}^{1/2}\log E_P\big(E_Q(x)\big) {\big(E_Q(x)\big)}^{1/2} \geq \log \lambda(Q,P) E_Q(x)  + {\big(E_Q(x)\big)}^{1/2} \log E_Q(x) {\big(E_Q(x)\big)}^{1/2}.$$ Thus,
 $$\tr\bigg(E_P\big(E_Q(x)\big)\log E_P\big(E_Q(x)\big)\bigg)\geq \log \lambda(Q,P) \tr(x) + \tr\bigg(E_Q(x)\log E_Q(x)\bigg).$$ In other words, we get
 $$\tr\bigg(\eta E_P\big(E_Q(x)\big)\bigg)-\tr\bigg(\eta E_Q(x)\bigg)\leq -\log\lambda(Q,P) \tr(x).$$ 
 As $\eta E_P\big(E_Q(x)\big)\geq E_P\big(E_Q(\eta x)\big)$ (see \cite[page 74]{PP}, for instance)  we readily obtain $$\tr(\eta x)-\tr (\eta E_Q(x))\leq -\log\lambda(Q,P) \tr(x).$$ So,
 $$\big(\tr(\eta x)-\tr (\eta E_P(x))\big) + \big(\tr\big(\eta E_P(x)\big) -\tr(\eta E_Q(x))\big)\leq -\log\lambda(Q,P) \tr(x).$$ Summing up over a partition of unity $(x_i)$ in $ M_{+}$ and taking the supremum over all such partitions we get $H(P|M)+H(Q|P)\leq -\log \lambda(Q,P).$  But as $P\subset M$, it follows easily that $H(P|M)=0$ (see \cite[page 75]{PP}). Thus we have proved that 
 \begin{equation}\label{onedirection}
 H(Q|P)\leq -\log \lambda(Q,P).
 \end{equation}
 \smallskip
 
 We now prove the non-trivial implication. The proof is inspired by \cite{PP} (see also \cite{NS}). 
 Note, $e^N_{Q_{-1}}$ is a projection in $Q$ with $\tr(e^N_{Q_{-1}})=\frac{1}{[Q:N]}$. Put $x=e^N_{Q_{-1}}-\frac{1}{[Q:N]}1.$ Then,
  $\tr(x)=0.$ Following \cite[Lemma 4.2]{PP}, let
  $$K_x=\overline{\text{conv}\{vxv^*:v\in \mathcal{U}(N)\}}.$$ A similar calculation as in \cite[Lemma 4.2]{PP} shows that $0\in K_x$ and hence for any fixed $\epsilon>0$ there are unitaries
  $v_1,\ldots, v_n \in N$ such that 
  $$\big\lVert \sum_i \frac{1}{n} v_ixv^*_i\big\rVert_2<{\epsilon}^2\frac{1}{[Q:N]}.$$  Therefore, we see that
  $$\big\lVert \sum_i \frac{[Q:N]}{n} v_ie^N_{Q_{-1}}v^*_i-1\big\rVert_2<{\epsilon}^2.$$
  Put $y=\frac{[Q:N]}{n}\sum_iv_ie^N_{Q_{-1}}v^*_i$. Then,
  \begin{equation}\label{yminusone}
 \lVert y-1\rVert_2<\epsilon^2.\end{equation}
  Let $p$ be the spectral projection of $y$ corresponding to the interval $[0,1+\epsilon].$  Using the following obvious inequality  $$\frac{1}{(1+\epsilon)}t\mathcal{X}_{[0,1+\epsilon]}(t)\leq 1,$$ and then applying functional calculus we easily obtain $\frac{1}{(1+\epsilon)}yp\leq 1.$
Now, put
  $$x_i= \frac{[Q:N]}{(1+\epsilon)n} v_ie^N_{Q_{-1}}v^*_i\wedge p.$$
  
  Thus, $$\sum_i x_i\leq \frac{[Q:N]}{(1+\epsilon)n}\sum_i pv_ie^N_{Q_{-1}}v^*_ip\leq \frac{1}{(1+\epsilon)}yp\leq 1.$$
  
  Now, it follows immediately from \Cref{cs} that 
  \begin{equation}\label{H}
  H(Q|P)\geq \sum_i \tr\big(\eta(E_P(x_i))-\eta(E_Q(x_i))\big).
  \end{equation}
  Since, for any scalar $\alpha, t > 0$ we have $\eta(\alpha t)=\eta(\alpha)t+\alpha \eta(t)$, a simple calculation yields 
$$ \sum_i \tr\big(\eta E_P(x_i)-\eta E_Q(x_i)\big) = \frac{[Q:N]}{(1+\epsilon)n}\sum_i \tr\big(\eta E_P(v_ie^N_{Q_{-1}}v^*_i\wedge p)- \eta E_Q(v_ie^N_{Q_{-1}}v^*_i\wedge p)\big);$$ and so, by Inequality \ref{H} we get
  
\begin{equation}\label{H2} H(Q|P)\geq \sum_i \tr\big(\eta E_P(x_i)-\eta E_Q(x_i)\big) \geq \frac{[Q:N]}{(1+\epsilon)n}\sum_i\tr\big(\eta(E_P(v_ie^N_{Q_{-1}}v^*_i\wedge p))\big).\end{equation}
 By \cite[page 74]{PP} we know that if $a,b\in M_{+},$ then $\tr (\eta(a+b))\leq \tr(\eta a) +\tr(\eta b)$ and so,
 \begin{equation*}
 \tr\bigg(\eta\big(E_P(v_ie^N_{Q_{-1}}v^*_i\wedge p)\big)\bigg)\geq \tr\bigg(\eta\big(E_P(v_ie^N_{Q_{-1}}v^*_i)\big)\bigg)-\tr\bigg(\eta\big(E_P(v_ie^N_{Q_{-1}}v^*_i)-E_P(v_ie^N_{Q_{-1}}v^*_i\wedge p)\big)\bigg).\end{equation*}
  By concavity of $\eta$, it is well-known  that $\eta(v^*xv)\geq v^*\eta(x)v$ for any $v$ with $\lVert v\rVert \leq 1$ and  $x \in M_{+}$ with $\lVert x\rVert  \leq 1$ (see \cite[B.1]{NS}, for instance). Therefore, it follows that $$\tr\bigg(\eta\big(E_P(v_ie^N_{Q_{-1}}v^*_i)\big)\bigg)=\tr\bigg(\eta\big(v_iE_P(e^N_{Q_{-1}})v^*_i\big)\bigg) \geq \tr\bigg(v_i\eta\big(E_P(e^N_{Q_{-1}})\big)v^*_i\bigg)=\tr\bigg(\eta\big(E_P(e^N_{Q_{-1}})\big)\bigg).$$
  Thus, by Inequality \ref{H2}, we obtain
  \begin{equation}\label{H3}
  H(Q|P)\geq  \frac{[Q:N]}{(1+\epsilon)n}\bigg\{\sum_i\tr\big(\eta(E_P(e^N_{Q_{-1}}))\big)- \tr\bigg(\eta\big(E_P(v_ie^N_{Q_{-1}}v^*_i)-E_P(v_ie^N_{Q_{-1}}v^*_i\wedge p)\big)\bigg)\bigg\}.
  \end{equation}
 Now, a similar calculations as in the proof of \cite{PP}[Lemma 4.2] yields
 \begin{equation}\label{H4}
  \tr\bigg(\eta\big(E_P(v_ie^N_{Q_{-1}}v^*_i)-E_P(v_ie^N_{Q_{-1}}v^*_i\wedge p)\big)\bigg)\leq \eta(\epsilon^2). \end{equation} We provide the details for the convenience of the reader. In view of the fact that $\tr\circ \eta \leq \eta\circ \tr$, it is sufficient to prove that 
  \begin{equation}\label{H5}
  \eta\big(\tr(v_ie^N_{Q_{-1}}v^*_i-v_ie^N_{Q_{-1}}v^*_i\wedge p)\big)\leq \eta(\epsilon^2).\end{equation}Since $\tr(e\vee f)=\tr(e)+\tr(f)-\tr(e\wedge f)$ for any projections $e$ and $f$ in $M$ we get
  $$\tr(v_ie^N_{Q_{-1}}v^*_i-v_ie^N_{Q_{-1}}v^*_i\wedge p)\leq 1-\tr(p).$$
 Applying functional calculus we easily see that $y(1-p)\geq (1+\epsilon) (1-p)$ and so
  $${\lVert y-1\rVert}_2^2\geq \tr\big((y-1)^2(1-p)\big)\geq \epsilon^2\tr(1-p).$$ Thus, by Inequality \ref{yminusone}, we obtain $1-\tr(p)\leq {\epsilon}^{-2}{\lVert y-1\rVert}_2^2\leq {\epsilon}^2$. So we obtain
  $$\tr(v_ie^N_{Q_{-1}}v^*_i-v_ie^N_{Q_{-1}}v^*_i\wedge p)\leq {\epsilon}^2.$$
  As $\eta$ is increasing in $[0,{\epsilon}^2]$ for sufficiently small $\epsilon$, we see that the Inequality \ref{H5} is now obvious. This proves the truth of the Inequality \ref{H4}.
  \smallskip
  
Below we show that, using Inequalities \ref{H3} and \ref{H4}, the following inequalities hold true:
  
  \begin{align*}
   H(Q|P)\geq & \frac{[Q:N]}{(1+\epsilon)n}\sum_i \tr\bigg(\eta\big(E_P(e^N_{Q_{-1}})\big)\bigg)-\frac{[Q:N]}{(1+\epsilon)} \eta(\epsilon^2)\\
    = & -\frac{[Q:N]}{(1+\epsilon)}\frac{1}{[Q:N]}\log(\lambda(Q,P))-\frac{[Q:N]}{(1+\epsilon)} \eta(\epsilon^2)~~~~~~[\text{Using~~~\Cref{liu}~~~]}\\
    \geq & -\frac{1}{(1+\epsilon)}\log(\lambda(Q,P))-\frac{[Q:N]}{(1+\epsilon)} \eta(\epsilon^2).
  \end{align*}

Therefore, letting $\epsilon\rightarrow 0$  we conclude that
\begin{equation}\label{otherdirection}
H(Q|P)\geq -\log\big(\lambda(Q,P)\big).\end{equation}
The proof is now complete once we combine Inequality \ref{onedirection} and Inequality \ref{otherdirection}.\color{black}
 \end{proof}
 Applying \Cref{imp1} we immediately obtain the following (possibly useful) formula for the relative entropy.
 \begin{corollary}\label{imp2}
  Let $(N,P,Q,M)$ be an irreducible quadruple such that $[M:N]<\infty$. Then,
  $$H(Q|P)= \log\big(\tr(e_Q)\big)-\log\big(\tr(e_Pe_Q)\big)$$ and $$H(P|Q)= \log\big(\tr(e_P)\big)-\log\big(\tr(e_Pe_Q)\big)$$
 \end{corollary}

\begin{corollary}
 Let $(N,P,Q,M)$ be a quadruple of type $II_1$ factors with $N^{\prime}\cap M=\C$ and $[M:N]<\infty$. Then, $H(Q_1,P_1)= H(P|Q)$ and
 $H(P_1|Q_1)=H(Q|P).$
\end{corollary}
\begin{proof}
 By \Cref{main} we obtain $H(P_1|Q_1)=-\log\big(\lambda(P_1,Q_1)\big).$ Using \Cref{dual} we get
 $$H(P_1|Q_1)=-\log\big(\lambda(Q,P)\big).$$ Interchanging $P$ and $Q$ in the above equation we prove the other implication. This completes the proof.
\end{proof}
Similarly, applying \Cref{downwarddual} we get the following.
\begin{corollary}
 Let $(N,P,Q,M)$ be a quadruple of type $II_1$ factors with $N^{\prime}\cap M=\C$ and $[M:N]<\infty$. Then, $H(Q_{-1}|P_{-1})= H(P|Q)$ and
 $H(P_{-1}|Q_{-1})=H(Q|P).$
\end{corollary}
The following consequence is obvious once we  apply \Cref{main} and \Cref{prop}.

\begin{corollary}\label{prop2}
 Let $(N,P,Q,M)$ be a quadruple of type $II_1$ factors with $N^{\prime}\cap M=\C$ and $[M:N]<\infty$. If $[P:N]=2$ then $H(P,Q)$ is either $0$ or $\log 2.$
\end{corollary}
We conclude the paper with a remark. 

\begin{remark}
We believe that $H(P|Q)$ is a powerful invariant to investigate the relative position between intermediate subfactors and therefore, it is desirable to know the possible values of it.  We remark that \Cref{main} is a modest step towards achieving this goal. In future we want to dig deep into these entropic aspects of intermediate subfactor theory.
\end{remark}
\color{black}
\section{Acknowledgement}
I would like to thank Zhengwei Liu for various useful discussions and pointing out an error in an earlier version.


\begin{thebibliography}{ABC}

\bibitem{BDLR} K.~C.~Bakshi, S.~Das, Z.~Liu and Y.~ Ren, An angle between intermediate subfactors and its rigidity,
Trans. Amer. Math. Soc. {\bf 371} (2019), 5973-5991. 


\bibitem{BG} K.~C.~Bakshi and V.P. Gupta, A note on irreducible
  quadrilaterals of $II_1$-factors. {\em Internat. J. Math.}  {\bf 30
  } (2019), no. 12, 1950061, 22pp. 
  
\bibitem{CS} A. Connes, E. St\o rmer, Entropy for automorphisms of type $II_1$ von Neumann algebras, {\em Acta Math.} {\bf 134} (1975), 289--306.

   \bibitem{JLW} C.~Jiang, Z.~Liu and J.~Wu, Noncommutative uncertainty
    principles, J. Funct. Anal. {\bf 270} (2016), no.~1,
    264-311.
    
    \bibitem{Jo} V. F. R. Jones, Index for subfactors, Invent. Math. {\bf 72} (1983), no. 1, 1–25.

\bibitem{NS} S. Neshveyev and E St\o rmer, Dynamical entropy in operator algebras, {\em Results in
Mathematics and Related Areas 3rd Series A Series of Modern Surveys in Mathematics.} {\bf 50} (2006), Springer-Verlag, Berlin.
\bibitem{O} R. Okayasu, Relative entropy for abelian subalgebras, {\em Internat. J. Math.} {\bf 21} (2010), no. {\bf4}, 537-550.
\bibitem{PP} M.~Pimsner and S.~Popa, Entropy and index for
  subfactors, {\em Ann. Sci. Ecole Norm. Sup. (4)} {\bf 19} (1986),
  no.~1, 57-106.
\end{thebibliography}
\end{document}